\documentclass[dvipsnames,a4paper,twoside]{article}
\usepackage{amsthm,amsmath,amssymb}
\usepackage{hyperref}

\theoremstyle{plain}
\newtheorem{theorem}{Theorem}[section]

\theoremstyle{definition}

\newtheorem{example}[theorem]{Example}

\theoremstyle{remark}
\newtheorem{remark}{Remark}

\usepackage{graphicx,subfig}
\usepackage{xspace}
\usepackage[ulem=normalem,draft]{changes}

\newcommand{\Pb}{\mbox{\rm (P)}\xspace}
\newcommand{\Pbs}{\mbox{\rm (P$_\sigma$)}\xspace}
\newcommand{\U}{\text{\rm U}_{ad}}

\newcommand{\MT}{{\mathcal{M}[0,T]}}
\newcommand{\Y}{{\mathcal{Y}}}

\newcommand{\Ys}{\mathcal{Y}_\sigma}

\newcommand{\delay}{s}
\newcommand{\weight}{\kappa}
\newcommand{\shift}{\varsigma}

\title{Optimal time delays in a class of reaction-diffusion equations
\thanks{The first two authors were partially supported by the Spanish Ministerio de
Econom\'{\i}a y Competitividad under projects MTM2014-57531-P and MTM2017-83185-P.
The third author was supported by the collaborative research center
SFB 910, TU Berlin, project  B6.}}

\author{Eduardo Casas\thanks{Departmento de Matem\'{a}tica Aplicada y Ciencias de la Computaci\'{o}n, E.T.S.I. Industriales y de Telecomunicaci\'on, Universidad de Cantabria, 39005 Santander, Spain, {\tt eduardo.casas@unican.es}.}
\and Mariano Mateos\thanks{Departamento de Matem\'{a}ticas, Campus de Gij\'on, Universidad de Oviedo, 33203, Gij\'on, Spain, {\tt mmateos@uniovi.es}.}
\and Fredi Tr\"oltzsch\thanks{Institut f\"ur Mathematik, Technische Universit\"at Berlin, D-10623 Berlin,
Germany,  {\tt troeltzsch@math.tu-berlin.de.}}}

\begin{document}

\maketitle

\begin{abstract}A class of semilinear parabolic reaction diffusion equations with multiple time delays is considered. These time delays and corresponding weights
are to be optimized such that the associated solution of the delay equation is the best approximation of a  desired state function.  The differentiability of the mapping is proved that associates the solution of the delay equation
to the vector of weights and delays. Based on an adjoint calculus, first-order necessary optimality conditions are derived. Numerical test examples
show the applicability of the concept of optimizing time delays.
\end{abstract}

\begin{quote}
\end{quote}

\begin{quote}
\end{quote}

\section{Introduction}
\label{S1}

In this paper, we consider the optimization of Pyragas type feedback controllers in reaction-diffusion equations with
respect to finitely many time delays. The simplest example of an associated optimization problem is the following:
Let the semilinear parabolic equation with time delay $\delay \ge 0$
\begin{equation} \label{E1.1}
\frac{\partial}{\partial t}y(x,t)  - \Delta y(x,t) + R(y(x,t)) = \weight \, (y(x,t-\delay)-y(x,t)), \qquad (x,t) \in Q
\end{equation}
be given in $Q := \Omega \times (0,T),$ where $\Omega \subset \mathbb{R}^d$, $d \le 3$, is a bounded Lipschitz domain.
The equation is complemented by homogeneous Neumann boundary conditions and associated initial conditions.

Find a time delay $\delay$ and a weight $\weight \in \mathbb{R}$ such that the associated state function $y$ minimizes the distance to a desired state function $y_Q$ in the norm of $L^2(Q)$. In particular, we directly  optimize, say ''control'', the time delay $\delay$.

The optimization with respect to finitely many time delays and weights, associated first-order necessary optimality conditions, and numerical tests constitute the main novelty of our paper.

In view of the needed differentiability of
the mapping $\delay \mapsto y$,  the theory of optimality conditions turns out to be quite delicate. This differentiability issue was investigated first by Hale and Ladeira in \cite{Hale_Ladeira_1991} for ordinary  differential equations and
in \cite{Hale_Ladeira_1993} for nonlinear reaction-diffusion equations. They proved a version that is local in time, since under their assumptions the solution $y$ could blow up in finite time.
By a different method including certain monotonicity arguments, we were able to prove a general result on existence and uniqueness  for nonlocal reaction-diffusion equations
including measures in \cite{CMT2018}. This result is valid for arbitrary time horizons $T > 0$ and includes the equations considered here. Having this at our disposal, the proof of differentiability with respect to time delays  {became possible for arbitrary $T > 0$.}

More generally, we will consider multiple time delays $\delay_i$ and associated weights $\weight_i$, $i = 1,\ldots,m$, cf. equation \eqref{E1.2} below. To our best knowledge,  the optimization with respect to time delays  $\delay_i$  and associated weights $\weight_i$ was not yet investigated in literature. Compared with optimal control problems, the time delay $\delay$ and the weight $\weight$ play the role of the control, while $y$ is the state function of the control system. Although $u =(\delay,\weight)$ is not a control in the standard sense, we will occasionally call this vector a control.

This question might be interesting for applications. For instance, in laser technology, feedback controllers of Pyragas type are considered. Here, a laser beam is partially reflected by a semi-permeable mirror and
the reflected part is fed back after some time delay $\delay$. More general, a finite number of mirrors
can be used giving rise to finitely many time delays $\delay_1, \ldots, \delay_m$. Then, instead of \eqref{E1.1},
the more general equation
\begin{equation} \label{E1.2}
\frac{\partial}{\partial t}y(x,t)  - \Delta y(x,t) + R(y(x,t)) = \sum_{i=1}^m \weight_i \, y(x,t-\delay_i)
\end{equation}
is of interest, where the vectors $\delay = (\delay_1,\ldots,\delay_m)$ and $\weight  = (\weight_1,\ldots,\weight_m)$ are at our disposal. For Pyragas type problems with single or multiple time delays, the reader is referred to \cite{pyragas1992,pyragas2006}, the survey volume \cite{schoell_schuster2008}, and exemplarily to the papers \cite{kyrychko_blyuss_schoell11,siebert_alonso_baer_schoell14,siebert_schoell14}.

Our optimization problems with respect to the equation \eqref{E1.2} are somehow intermediate between the ones in our former contributions \cite{nestler_schoell_troeltzsch2016} and \cite{CMT2018} that investigate
the optimization of feedback kernels in nonlocal reaction-diffusion equations. In \cite{nestler_schoell_troeltzsch2016},
a nonlocal Pyragas type control system of the form
\begin{equation} \label{E1.3}
\frac{\partial}{\partial t}y(x,t)  - \Delta y(x,t) + R(y(x,t)) = \weight \, \left[ \int_0^T g(\delay) y(x,t-\delay) \, d\delay - y(x,t)\right]\end{equation}
is considered, where the kernel function $g$ is to be optimized, i.e. it plays the role of a  control. Later, in \cite{CMT2018}, we allowed measures as controls so that, in particular, Dirac measures could appear,
\begin{equation} \label{E1.4}
\frac{\partial}{\partial t}y(x,t)  - \Delta y(x,t) + R(y(x,t)) = \int_0^T y(x,t-\delay) \, d\mu(\delay),
\end{equation}
where the control $\mu$ is a regular Borel measure on $[0,T]$.

Our control system cannot be subsumed as a particular case of \eqref{E1.4}. In \eqref{E1.4}, the measure $\mu$ can be
composed of an absolutely continuous part (that is somehow related to $g$ in \eqref{E1.3}) and a singular part that
can be a combination of Dirac measures.  There is no a priori information on how the structure is, how many Dirac
measures appear, and where they are concentrated. In this sense,  \eqref{E1.4} is much more general than \eqref{E1.2}. On the other hand, the optimization of \eqref{E1.2} is restricted to a subset of the admissible controls for \eqref{E1.4}; in \eqref{E1.2}  the measure $\mu$ is required  to be a linear combination of $m$ Dirac measures $\delta_{\delay_i}$, $i \in \{1,\ldots,m\}$, with $m$ fixed.

This restriction to finitely many Dirac measures might be dictated by the technical background. In the application to Laser technology mentioned above,  the number
of semi-permeable mirrors might be   fixed for a given construction. Another application comes from medical science. For instance,
in Holt and Netoff \cite{Holt_Netoff2014}, linear combinations of {a fixed number of} Dirac measures
 are used in experiments that are related to the treatment of Parkinson's disease.

Our paper is organized as follows: In Section \ref{S2}, {we define and analyze our optimization problem. First,} we prove the differentiability of the mapping $\delay \mapsto y_\delay$. In principle, this differentiability is known from \cite{Hale_Ladeira_1993}. However,
in the setting of \cite{Hale_Ladeira_1993}, the existence of $y$ is only known locally in an interval $[0,\alpha)$. At $\alpha$, the solution  $y$ can blow up. A new version of the Banach fixed-point theorem was applied to prove differentiability. In our case, thanks to certain monotonicity properties, we have global existence on any interval
$[0,T]$ and are able to prove differentiability of the control-to-state mapping. Then, the existence of a solution and the optimality conditions is addressed. Section \ref{S3} is devoted to the numerical discretization of the problem. In Section \ref{S4} we present some numerical examples that show the applicability of our concept of controlling time delays.

\section{Analysis of the optimization problem}
\label{S2}

In this work, $\Omega$ is a domain of $\mathbb{R}^d$, $d\leq 3$, with Lipschitz boundary $\Gamma$,  while $T>0$ is a fixed final time; we will write $Q=\Omega\times(0,T)$ and  $\Sigma=\Gamma\times(0,T)$.
Moreover, we fix $m\in\mathbb{N}$, real parameters $0\leq a_i \leq b_i$, $i=1,\ldots,m$, and set $b=\max\{b_i:\ i=1,\ldots,m\}$  and $Q^-= \Omega\times(-b,0)$. We assume that $b < T$.

The initial data are defined in $\bar Q^-$ by a continuous function $y_0: \bar Q^- \to \mathbb{R}$. The reaction term is given by a function $R:Q\times\mathbb{R} \to \mathbb{R}$.  The assumptions on $\Omega$, $y_0$, and $R$ will be detailed later.

Finally, we introduce the admissible set
\[
\U=\{u=(\delay,\weight)\in\mathbb{R}^m\times\mathbb{R}^m:\ a_i\leq\delay_i\leq b_i,  \ {\alpha_i \le \kappa_i \le \beta_i,} \ 1 \le i \le m\},
\]
{where $-\infty \le \alpha_i \le \beta_i \le \infty$, $i = 1,\ldots, m$, are given real numbers.}

We consider the optimization problem
\[\Pb\qquad \min_{u\in \U}J(u)=\frac{1}{2}\int_Q(y_u-y_Q)^2\,dxdt+\frac{\nu}{2}|\weight|^2,\]
where {$\nu \ge 0$}, {$|\kappa|$ denotes the Euclidean norm of $\kappa$ in $\mathbb{R}^m$,} and $y_u$ is the unique solution of the state equation \eqref{E2.1} below,
\begin{equation}\label{E2.1}
\left\{\begin{array}{rcll}
\partial_t y -\Delta y +R(x,t,y) &=& \displaystyle\sum_{i=1}^m \weight_i y(x,t-\delay_i)&\mbox{ in }Q\\
\partial_n y & = & 0&\mbox{ on }\Sigma\\
y(x,t) & = & y_0(x,t)&\mbox{ in }Q^-.
\end{array}
\right.
\end{equation}
By $\partial_n$, we denote the outward normal derivative on $\Gamma$.

Notice that the right hand side of \eqref{E2.1} can be written as
 \[\displaystyle\sum_{i=1}^m \weight_i y(x,t-\delay_i) = \displaystyle\int_{[0,T]} y(x,t-s) d\mu(s)\]
with  $\mu = \displaystyle\sum_{i=1}^m\weight_i \delta_{\delay_i}\in \MT.$

{Let us mention that the right-hand side of \eqref{E2.1} is more general than a standard Pyragas feedback
as in equation \eqref{E1.1} that includes the term $-y(x,t)$ in the right-hand side. This term is obtained in \eqref{E2.1} by the particular delay $s_1 = 0$ with a suitable coefficient.}
\vspace{1ex}

We impose the following assumptions on the given data in \Pb.
\begin{itemize}
  \item[(A1)] The domain $\Omega$ is $W^{2,q}$ regular for some  $q>\frac{d}{2}+1$, i.e., if $y\in H^1(\Omega)$, $\Delta y\in L^{q}(\Omega)$ and $\partial_n y\in W^{1-1/q,q}(\Gamma)$, then $y\in W^{2,q}(\Omega)$.
  \item[(A2)] We require $y_0\in C(\bar Q^-)\cap W^{1,q}(-b,0;L^{q}(\Omega))$ and $y_0(\cdot,0)\in W^{2-\frac{2}{q},q}(\Omega)$.
  \item[(A3)] $R$ is a Carath\'eodory function of class $C^1$ with respect to the last variable such that
 \begin{align*}
 &  R(\cdot,\cdot,0) \in L^{q}(Q),\vspace{2mm}\\
 & \exists C_R\in\mathbb{R} : \partial_y R(x,t,y)\geq C_R,\quad \forall y \in \mathbb{R},\vspace{2mm}\\
 & \forall M > 0\ \exists C_M : |\partial_y R(x,t,y)| \leq C_M,\quad \forall |y| \le M,
 \end{align*}
holds  for almost all $(x,t) \in Q$.
 \end{itemize}

Notice that (A1) is satisfied in convex plane polygonal domains or in domains with boundary of class $C^{1,1}$.

We will consider the state space
\[\Y = C(\bar Q)\cap W^{2,1}_{q}(Q),\]
where $W^{2,1}_{q}(Q) = L^{q}(0,T;W^{2,q}(\Omega))\cap W^{1,q}(0,T;L^{q}(\Omega))$;
$\Y$ is a Banach space endowed with the usual intersection norm.

\begin{theorem}\label{T2.1}Under assumptions (A1), (A2), and (A3), for every ${u}\in \U$ there exists a unique solution $y_u\in \Y$ of \eqref{E2.1}. Moreover, for all $r > 0$ there exists a constant $C_r$ such that
\[
\|y_u\|_{\Y}\leq C_r \left(\|y_0\|_{C(\bar Q_-)}\sum_{i = 1}^m|\weight_i| + \|y_0(\cdot,0)\|_{W^{2-\frac{2}{q},q}(\Omega)} + \|R(\cdot,\cdot,0)\|_{L^{q}(Q)}\right)
\]
holds for all $u = (\delay,\weight) \in \mathbb{R}^m \times \mathbb{R}^m$ with $|\weight| \le r$.
\end{theorem}
\begin{proof}Existence and uniqueness of the solution $y_u\in C(\bar Q)\cap L^2(0,T;H^1(\Omega))$ follow from \cite[Th. 2.2]{CMT2018} with $u = \sum_{i=1}^m \weight_i \delta_{\delay_i}$, where the following estimate is proved
\[
\|y_u\|_{C(\bar Q)} \le c_r \left(\|y_0\|_{C(\bar Q_-)} {\|u\|_{\MT}} + \|y_0(\cdot,0)\|_{C(\bar\Omega)} + \|R(\cdot,\cdot,0)\|_{L^{q}(Q)}\right).
\]
Notice that $\|u\|_{\MT}= \sum_{i = 1}^m|\weight_i|$. Once this is obtained, from \eqref{E2.1} and assumptions (A1)-(A3) we infer $\partial_t y_u-\Delta y_u \in  L^{q}(Q)$. Now, the $W_{q}^{2,1}(Q)$ regularity and the corresponding estimate  follows from \cite[Th. IV.9.1]{Lad-Sol-Ura68} and the inequality established above.
\end{proof}

We mention that the function $\tilde{y}_u$, defined by
\[\tilde y_u(x,t)=\left\{\begin{array}{cc}
                           y_u(x,t), & \ \ 0\leq t \leq T,\\
                           y_0(x,t), & -b\leq t < 0,
                         \end{array}
\right.
\]
belongs to $W^{1,q}(-b,T;L^{q}(\Omega))$. This is a consequence of the regularity established in the theorem and assumption (A2). In what follows, when this does not lead to confusion, we will identify $y_u$ with its extension $\tilde y_u$.
\vspace{1ex}

{By the next result,} we improve the differentiability result of \cite{Hale_Ladeira_1993}.

\begin{theorem}\label{T2.2} The control-to-state mapping $G:\U\to \Y$, $u \mapsto y_u$ {has partial derivatives } $\partial_{\delay_i}{G(u)}$ and
$\partial_{\weight_i}G(u)$
 given as follows: For every ${u}\in \U$ and $1 \le i \le m$, we have $\partial_{\delay_i}{G(u)} = z_i$ where $z_i$ satisfies the equation
\begin{equation}\label{E2.2}
\left\{\begin{array}{l}
\partial_t z -\Delta z +\partial_y R(x,t,{y_u})z =  \displaystyle\sum_{j=1}^m \weight_j z(x,t-\delay_j) - \weight_i\partial_t {y_u}(x,t-\delay_i) \ \mbox{ in }Q\vspace{2mm}\\
\partial_n z = 0\ \mbox{ on }\Sigma,\ z = 0 \ \mbox{ in }Q^-,
\end{array}
\right.
\end{equation}
and $\partial_{\weight_i}G(u) = \eta_i$, where $\eta_i$ satisfies
\begin{equation}\label{E2.3}
\left\{\begin{array}{l}
\partial_t \eta -\Delta \eta +\partial_y R(x,t,y_u)\eta =  \displaystyle\sum_{j=1}^m \weight_j \eta(x,t-\delay_j) + y_u(x,t-\delay_i)\ \mbox{ in }Q\vspace{2mm}\\
\partial_n \eta = 0\ \mbox{ on }\Sigma,\ \eta = 0 \ \mbox{ in }Q^-.
\end{array}
\right.
\end{equation}
\end{theorem}

\begin{proof}
We fix $u = (\delay,\weight) \in \U$ and write $y=G(u) = G(\delay,\weight)$. First, we calculate the partial derivative with respect to $\delay_i$. For sufficiently small $|\rho|$, we write $y_\rho=G(\delay+\rho e_i,\weight)$, where $e_i$ denotes the $i$-th vector of the canonical base of $\mathbb{R}^m$. We have to compute
\[\partial_{\delay_i} G(\delay,\weight) =\lim_{\rho\to 0}\frac{y_\rho-y}{\rho},\]
where the limit is restricted to $\rho>0$ if $\delay_i=a_i$ and to $\rho <0$ if $\delay_i=b_i$, since we have to determine the right and left derivatives in these points, respectively. Define $z_\rho = \frac{y_\rho-y}{\rho}$; subtracting the partial differential equations and dividing by $\rho$ we get by the mean value theorem for $\hat y_\rho(x,t) = y(x,t)+\theta(x,t)(y_\rho(x,t)-y(x,t))$, $0<\theta(x,t)<1$,
\begin{align}
&\partial_t z_\rho -\Delta z_\rho + \partial_yR(x,t,\hat y_\rho )z_\rho\nonumber\\
&  = \sum_{j \neq i}\weight_j\frac{y_\rho(x,t-\delay_j)-y(x,t-\delay_j)}{\rho}
 + \weight_i\frac{y_\rho(x,t-\delay_i - \rho)-y(x,t-\delay_i)}{\rho}\nonumber\\
 & = \sum_{j\neq i}\weight_j z_\rho(x,t-\delay_j) +{\weight_i z_\rho(x,t-s_i-\rho)}+ \weight_i\frac{y(x,t-\delay_i - \rho)-y(x,t-\delay_i)}{\rho}.\label{E2.4}
\end{align}
Using Theorem \ref{T2.1} and taking into account that $z(0) = 0$ in $\Omega$ and $\partial_n z = 0$ on $\Sigma$, we deduce
\begin{equation}\label{E2.5}
\begin{aligned}
\|z_\rho\|_{\Y} \leq C\,\left(\int_Q\left(\displaystyle\frac{y(x,t-\delay_i-\rho)-y(x,t-\delay_i)}{\rho} \right)^q\,dxdt\right)^{1/q}\\[1ex]
= C\,\left\| \displaystyle\frac{y(\cdot,\cdot-\delay_i-\rho)-y(\cdot,\cdot-\delay_i)}{\rho} \right\|_{L^{q}(Q)}
\end{aligned}
\end{equation}
with some constant $C>0$, which may depend on $\weight$, but is independent of $\rho$ { and $\delay$}.
Since $y\in W^{1,q}(0,T,L^{q}(\Omega))$, we have that
\begin{equation}
\lim_{\rho\to 0} \displaystyle\frac{y(x,t-\delay_i-\rho)-y(x,t-\delay_i)}{\rho} = -\partial_t y(x,t-\delay_i) \ \text{ in } L^{q}(Q).
\label{E2.6}
\end{equation}
Indeed, consider $\varepsilon>0$ arbitrary. Then, for all $|\rho|$ small enough, applying \cite[Thm 1.1 in page 57]{Necas67}, we obtain


  {
\begin{align*}
  &\left(\int_Q\left(\displaystyle\frac{y(x,t-\delay_i-\rho)-y(x,t-\delay_i)}{\rho} +\partial_t y(x,t-\delay_i)\right)^q\,dxdt\right)^{1/q}\\
  = & \left(\int_Q \left(-\int_0^1\left(\partial_t y(x,t-\delay_i-\lambda \rho)-\partial_t y(x,t-\delay_i)\right)d\lambda\right)^q\,dxdt\right)^{1/q}\\
    = & \left\|-\int_0^1\left(\partial_t y(\cdot,\cdot-\delay_i-\lambda \rho)-\partial_t y(\cdot,\cdot-\delay_i)\right)d\lambda\right\|_{L^{q}(Q)}\\
  \leq& \int_0^1 \left\|\partial_t y(\cdot,\cdot-\delay_i-\lambda \rho)-\partial_t y(\cdot,\cdot-\delay_i)\right\|_{L^{q}(Q)}d\lambda  \\
  < & \int_0^1\varepsilon d\lambda = \varepsilon.
  \end{align*}
 }

From \eqref{E2.5} and \eqref{E2.6}, we deduce that $\{z_\rho\}_\rho$ is uniformly bounded in $\Y$. Hence we can extract a subsequence that converges weakly in $\Y$ to some $z$. Since $\Y$ is compactly embedded in $L^{q}(Q)$, we also have that $z_\rho \to z$ strongly in $L^{q}(Q)$. {Since the right hand side of \eqref{E2.4} is bounded in $L^q(Q)$ and $y_0(\cdot,0)$ is a H\"older function in $\bar\Omega$, we have that there exists $\mu \in(0,1)$ such that $\{z_\rho\}_\rho$ is bounded in $C^{0,\mu}(\bar Q)$, see \cite[III-10]{Lad-Sol-Ura68}. Using that $C^{0,\mu}(\bar Q)$ is compactly embedded in $C(\bar Q)$, we have that $z_\rho\to z$ strongly in $C(\bar Q)$.} Passing to the limit in \eqref{E2.4}, in view of \eqref{E2.6}, we obtain \eqref{E2.2}.

Now, we calculate the partial derivative with respect to $\weight_i$. For small $|\rho|$, we define $y_\rho=G(\delay,\weight+\rho e_i)$ and $\eta_\rho = (y_\rho-y)/\rho$. As above, there exists $\hat y_\rho(x,t) = y(x,t)+\theta(x,t)(y_\rho(x,t)-y(x,t))$ with some measurable function $0<\theta(x,t)<1$ such that
\[
\partial_t \eta_\rho -\Delta \eta_\rho + \partial_yR(x,t,\hat y_\rho )\eta_\rho  =
\sum_{j = 1}^m\weight_j\eta_\rho(x,t-\delay_j) + y_\rho(x,t-\delay_i).
\]
Again $\{\eta_\rho\}_\rho$ is uniformly bounded in {$\Y\cap C^{0,\mu}(\bar Q)$ for some $\mu>0$}, and we can pass to the limit to obtain \eqref{E2.3}.
\end{proof}

By Theorem \ref{T2.2} and the chain rule, the functional $J$ is differentiable and its derivative has the following form.

\begin{theorem}\label{T3.3}The functional $J$ {has partial derivatives}
\begin{align}
&\frac{\partial J}{\partial \delay_i}(u) = - \weight_i \int_Q\varphi_u(x,t)\partial_t y_u(x,t-\delay_i) \,dxdt,\label{Dwrtdelay}\\
&\frac{\partial J}{\partial \weight_i}(u) = \nu \weight_i +\int_Q\varphi_u(x,t) y_u(x,t-\delay_i) \,dxdt,\label{Dwrtweight}
\end{align}
for $1 \le i \le m$, where the adjoint state $\varphi_{u}\in \Y$  is the unique solution to the {advanced} adjoint equation
\begin{equation}\label{E2.7}
\left\{\begin{array}{l}
-\partial_t \varphi -\Delta \varphi + \partial_yR(x,t,y_u)\varphi = y_u-y_Q + \displaystyle\sum_{i=1}^m \weight_i \varphi (x,t+\delay_i)\ \mbox{ in }Q\vspace{2mm}\\
\partial_n \varphi(x,t) = 0\ \mbox{ on }\Sigma,\ \varphi(x,t) = 0\ \mbox{ if }t\geq T.
\end{array}
\right.
\end{equation}
\end{theorem}

\begin{proof}
Using the chain rule, we obtain
\[
\frac{\partial J}{\partial \delay_i}(u) = \int_Q(y_{u}-y_Q) z_i\,dxdt \quad \text{ and }\quad
\frac{\partial J}{\partial \weight_i}(u) = \int_Q(y_u-y_Q) \eta_i\,dxdt +\nu \weight_i,
\]
where $z_i\in\Y$ is the solution of \eqref{E2.2} and $\eta_i$ is the solution of \eqref{E2.3}.

Let us consider the derivative with respect to $\delay_i$. Using the adjoint state equation \eqref{E2.7}, integration by parts and the equation \eqref{E2.2} satisfied by $z_i$, we obtain
\begin{align*}
\int_Q(y_u&-y_Q) z_i \,dxdt = \\& \displaystyle\int_Q \big[-\partial_t\varphi_u-\Delta \varphi_u+ \partial_yR(x,t,y_u)\varphi_u-\sum_{j=1}^m \weight_j \varphi_u(x,t+\delay_j)\big] z_i \,dxdt\\
=&\displaystyle\int_Q \varphi_u\big[\partial_t z_i-\Delta z_i+\partial_yR(x,t,y_u) z_i- \sum_{j=1}^m\weight_j z_i(x,t-\delay_j)\big]\,dxdt \\
= & -\weight_i \displaystyle\int_Q \varphi_u (x,t)\partial_t y_u(x,t-\delay_i) \,dxdt.
\end{align*}
Here we performed the change of variables $\tilde t=t+\delay_j$ and took into account the final conditions satisfied by $\varphi_u$ along with the initial conditions satisfied by $z_i$ to write
\[\int_Q \varphi_u(x,t+\delay_j) z_i(x,t) dt dx = \int_Q \varphi_u(x,t) z_i(x,t-\delay_j) \, dtdx \]

The derivative with respect to $\weight_i$ is obtained in a similar way.
\end{proof}

Next, we show the well-posedness of $\Pb$.
\begin{theorem}\label{T2.4} {If $\nu > 0$ or $ -\infty < \alpha_i \le \beta_i < \infty$ for all $i \in \{1,\ldots,m\}$, then}
Problem $\Pb$ has a solution $\bar u = (\bar\delay,\bar\weight)$.
\end{theorem}
\begin{proof}If $u^k=(\delay^k,\kappa^k)\to u=(\delay,\weight)$ in $\mathbb{R}^m\times \mathbb{R}^m$, then
$\sum_{i=1}^m \weight_i^k \delta_{\delay_{i}^k}\stackrel{*}{\rightharpoonup}\sum_{i=1}^m \weight_i \delta_{\delay_i}$
 in $\MT$ {as $k \to \infty$}. So following \cite[Lemma 3.2]{CMT2018}, we have that
${y_{u^k}}\to {y_u}$ strongly in $ L^2(0,T;H^1(\Omega))\cap C(\bar Q)$.
Therefore $J$ is continuous in $\U$ and obviously $\U$ is closed in $\mathbb{R}^m\times \mathbb{R}^m$.

{Thanks to our assumptions, either the objective functional is coercive or $\U$ is compact. Since we are dealing with a finite dimensional problem, it is clear that $\Pb$ has a global solution.}
\end{proof}

Now we are able to set up the first order necessary optimality conditions.
\begin{theorem}
Let ${\bar u}\in \U$ be a local solution of $\Pb$ and let $\bar y$ be the associated state defined by
\begin{equation}\label{E3.3}
\left\{\begin{array}{rcll}
\partial_t \bar y -\Delta \bar y +R(\bar y) &=& \displaystyle\sum_{j=1}^m \bar \weight_i \bar y(x,t-\bar \delay_i)&\mbox{ in }Q\\
\partial_n \bar y & = & 0&\mbox{ on }\Sigma\\
\bar y(x,t) & = & y_0(x,t)&\mbox{ in }Q^-.
\end{array}
\right.
\end{equation}
Then there exists a unique adjoint state  $\bar \varphi\in\Y$ such that the adjoint equation
\begin{equation}\label{E3.4}
\left\{\begin{array}{rcll}
-\partial_t \bar \varphi -\Delta \bar \varphi + \partial_yR(x,t,\bar y)\bar \varphi &=& \bar y-y_Q + \displaystyle\sum_{i=1}^m \bar \weight_i \bar\varphi (x,t+\bar \delay_i)&\mbox{ in }Q\\
\partial_n \bar \varphi & = & 0&\mbox{ on }\Sigma\\
\bar \varphi & = & 0&\mbox{ if }t\geq T,
\end{array}
\right.
\end{equation}
the variational inequalities
\begin{equation}
 -{\bar{\weight_i}} \int_Q \partial_t \bar y(x,t-\bar\delay_i)\, \bar \varphi(x,t) \,dxdt (\delay_i-\bar\delay_i)\geq 0\quad \forall \delay_i \in [a_i,b_i],
\end{equation}
{and
\begin{equation}
  {\left(\nu\bar \weight_i + \int_Q \bar y(x,t-\bar\delay_i)\, \bar \varphi(x,t) \,dxdt\right)(\weight_i-\bar\weight_i)\geq 0\quad\forall \weight_i \in [\alpha_i,\beta_i]\cap \mathbb R,}
\end{equation}
}are satisfied for $i=1,\ldots,m$.
\end{theorem}

\section{Numerical Discretization}\label{S3}
We suppose {that} $\Omega$ {is} polygonal or polyhedral and consider, cf. \cite[definition (4.4.13)]{Brenner-Scott2002}, a quasi-uniform family of triangulations $\{\mathcal{K}_{h}\}_{h>0}$ of $\bar\Omega$ and a {quasi-uniform}
family of  partitions of size $\tau$ of $[0,T]$, $0=t_0<t_1<\dots<t_{N_\tau}=T$. We {define}  $I_k=(t_{k-1},t_k]$, $\tau_k = t_k-t_{k-1}$, $\tau = \max\{\tau_k\}$, and {introduce the space-time mesh size} $\sigma=(h,\tau)$.

Now we consider the finite dimensional spaces
\[
\begin{aligned}
&Y_h = \{z_h\in C(\bar\Omega):\ z_{h|K}\in \mathcal{P}^1(K)\ \forall K\in\mathcal{K}_h\},\\[1ex]
&\Ys^0 =\{\phi_\sigma\in L^2(0,T;Y_h): \phi_{\sigma|I_k}\in \mathcal{P}^0(I_k;Y_h)\ \forall k =1,\dots,N_{\tau}\},\\[1ex]
&\Ys^1 =\{y_\sigma\in C([0,T];Y_h): y_{\sigma|I_k}\in \mathcal{P}^1(I_k;Y_h)\ \forall k =1,\dots,N_{\tau}\}
\end{aligned}
\]
{where $\mathcal{P}^1(K)$ is the set of polynomials of degree 1 in $K$ and, for $i=0,1$,  $\mathcal P^i(I_k;Y_h)$ is the set of polynomials of degree $i$ defined in $I_k$ with values in $Y_h$.}

For $\phi_\sigma\in\Ys^0$, we denote {by} $\phi_\sigma^k\in Y_h$ the value of $\phi_\sigma$ in $I_k$. We also remark that $\Ys^1$ {is contained in} $W^{1,q}(0,T;L^{q}(\Omega))$ and, if $y_\sigma\in\Ys^1$, then $\partial_t y_\sigma$ can be identified with an an element of $\Ys^0$.

{The discrete state equation is defined {in a variational form} as follows: For given control vector $u = (\delay,\weight)$,
the associated discrete state $y_\sigma(u)\in \Ys^1$} is the unique solution of (cf. \cite[Eq. (23)]{Becker-Meidner-Vexler2007})
\begin{align}
& y_\sigma(x,0) = {\Pi}_h y_0(x,0),\quad 
\nonumber\\
&\int_Q \frac{\partial y_\sigma}{\partial t}\phi_\sigma \,dxdt+\int_Q \nabla_x y_\sigma\nabla_x\phi_\sigma \,dxdt +\int_Q R(x,t,y_\sigma)\phi_\sigma \,dxdt\label{DSE}
\\&\qquad =
\sum_{i=1}^{m}\weight_i\left[\int_0^{\delay_i}y_0(x,t-\delay_i)\phi_\sigma\,dxdt + \int_{\delay_i}^{T}y_\sigma(x,t-\delay_i)\phi_\sigma\,dxdt\right],\quad \forall\phi_\sigma\in\Ys^0,\nonumber
\end{align}
where $\Pi_h:L^2(\Omega)\to Y_h$ is the projection onto $Y_h$ in the $L^2(\Omega)$-sense.

The discretized {optimization} problem is
\[\Pbs\qquad\min_{u\in \U}J_\sigma(u)=\frac{1}{2}\int_Q \left(y_\sigma{(u)}(x,t)-y_Q(x,t)\right)^2 \,dxdt+\frac{\nu}{2} |\weight|^2.
\]

%
%
%

{{To compute} the partial derivatives of $J_\sigma$, we invoke an associated discrete adjoint equation.} For every $u\in \U$, we define the {associated} discrete adjoint state $\varphi_\sigma(u)\in \Ys^0$ {as} the unique solution of (cf. \cite[Eq. (25)]{Becker-Meidner-Vexler2007})
\begin{align}
& \varphi_\sigma^{N_\tau+1}=0\nonumber\\
& -\sum_{k=1}^{N_\tau} \int_\Omega z_\sigma(x,t_k) (\varphi_\sigma^{k+1}-\varphi_\sigma^{k})\, dx
   +\int_Q\nabla_x z_\sigma\nabla_x\varphi_\sigma\,dxdt\nonumber\\
& +\int_Q \partial_yR (x,t,y_\sigma(u))z_\sigma\varphi_\sigma\,  \,dxdt
=  \int_Q (y_\sigma(u)-y_Q)z_\sigma\,dxdt \label{DASE}\\
&+ \sum_{i=1}^{m}\weight_i \int_0^{T-\delay_i}\int_\Omega \varphi_\sigma(x,t+\delay_i)z_\sigma\,dxdt ,\quad \forall z_\sigma\in \Ys^1,\nonumber
\end{align}
where we have introduced an artificial $\varphi_\sigma^{N_\tau+1}$ to simplify the notation.

{Both the discrete state equation \eqref{DSE} and the discrete adjoint state equation \eqref{DASE} can be solved using a time-marching scheme. Despite the differences in the variational formulations, in both cases a Crank-Nicholson time-marching scheme  is obtained, cf. \cite[p. 824]{Becker-Meidner-Vexler2007}.}

\begin{remark}{Notice that the time instants $t_k$, $k = 1,\ldots N_\tau$, can be taken completely independent of the location of the time delays.
Moreover, the time delays can admit any value between $0$ and $b$; they also can coincide with some of the the $t_k$'s. Compared with standard Euler time stepping methods, this is an essential advantage of this numerical technique. }
\end{remark}
Now, with exactly the same technique used for problem \Pb, we can prove that $J_\sigma$ has partial derivatives
and that
\begin{align}
\nonumber\frac{\partial J_\sigma}{\partial_{\delay_i}} (u) =& -\weight_i\left[\int_0^{\delay_i}\int_\Omega \partial_t y_0(x,t-s_i) \varphi_\sigma(x,t)\,dxdt\right. \\
&\qquad\qquad+\left.\int_{\delay_i}^T\int_\Omega \partial_t y_\sigma(x,t-s_i) \varphi_\sigma(x,t)\,dxdt\right]\label{E3.1}\\
\nonumber\frac{\partial J_\sigma}{\partial_{\weight_i}} (u) =& \nu\weight_i + \int_0^{\delay_i}\int_\Omega y_0(x,t-s_i)\varphi_\sigma(x,t)\,dxdt \\
&+ \int_{\delay_i}^T\int_\Omega y_\sigma(x,t-s_i)\varphi_\sigma(x,t)\,dxdt.\label{E3.2}
\end{align}
{The proof of {existence of partial derivatives} can be done following the same steps as for the continuous case. In a first step, we {compute the partial derivatives of the discrete state} as in Theorem \ref{T2.2}. The key estimate \eqref{E2.5} is replaced by the stability estimates in \cite[Corollary 4.8]{Meidner-Vexler-2011}; the limit in \eqref{E2.6} is also valid, since $\Ys^1\hookrightarrow W^{1,q}(0,T;Y_h)$. Finally, we can pass to the limit in the linearized discrete equation taking into account that the discretization parameters $(h,\tau)$ are fixed, so we are working in a finite dimensional space. The expressions for the derivatives of the discrete functional follow from the chain rule {as in the proof of Theorem \ref{T3.3}.}}

\begin{remark}{In recent contributions to PDE control, discontinuous Galerkin  (dG) methods became
quite popular, \cite{Becker-Meidner-Vexler2007}.} {We are able} to discretize both the state equation and the adjoint state equation using the same set of discontinuous Galerkin elements dG(0), cf. \cite[Eqs. (18) and (20)]{Becker-Meidner-Vexler2007} and to derive expressions for the partial derivatives of the resulting discrete functional. {However, the partial derivatives of the discrete objective functional are not everywhere continuous. The reason is the following:}

{To simplify the exposition, suppose that $\tau_k=\tau$ for all $k=1,\ldots,N_\tau$. Then, the control-to-discrete-state mapping is not differentiable at the nodes of the time mesh. Notice that the technique used in Theorem \ref{T2.2} cannot be applied because the discrete states are piecewise constant in time and $\Ys^0\not\hookrightarrow W^{1,q}(0,T;Y_h)$, so the derivative with respect to time of the discrete state  can not be {identified with} a function in $L^{q}(Q)$. Taking advantage of the fact that we are dealing with a finite dimensional problem, the partial derivatives of the discrete state with respect to the delays can be computed for any $t\neq t_k$, but  jump discontinuities will appear at the nodes of the time mesh.}

 {These jump discontinuities are inherited by the partial derivatives of the discrete functional. The expressions we obtain for the derivatives of the discrete {functional} are formally the same as \eqref{E3.1} and \eqref{E3.2} if we identify the time derivatives of elements in $\Ys^0$ with combinations of Dirac measures centered at the time nodes. This leads to the discontinuities in the partial derivatives of $J_\sigma$ with respect to the delays.}
\end{remark}

\section{Examples}\label{S4}

{The aim of this section is to confirm that optimizing time delays in nonlinear parabolic delay equations is a useful concept. In particular, we demonstrate that oscillatory patterns can be achieved by an associated feedback control. In this way, our method is also
some contribution to the topic of {``learning controller''}.

In our test examples,
we do not restrict ourselves to problem \Pb. We will start with {an example}  for {a related} ordinary differential delay equation. They are covered by our parabolic problem as particular case. It might be useful to first solve
an ODE control problem and take the obtained result as initial guess for the solution of the associated PDE control problem. In addition, in the case of ordinary differential equations the graphs of the desired state and the computed optimal state can be graphically better compared.

Moreover, in the context of approximating periodic states of parabolic delay equations,
we also consider a problem with slightly changed ''shifted'' objective functional {as suggested in \cite{nestler_schoell_troeltzsch2016}; see examples \ref{Example3} and \ref{Example4} below}.

To perform the optimization {numerically}, we use the \textsc{Matlab} {code} \texttt{fmincon} with the option \texttt{('SpecifyObjectiveGradient',true)} {that needs the gradient of the function to be minimized.
This code uses subroutines for calculating the functions $u \mapsto J_\sigma(u)$ and
$u \mapsto \nabla J_\sigma(u)$. Both functions are evaluated by solving the discretized state equation and adjoint equation, respectively, according to the methods explained in the last section.}

Since the code \texttt{fmincon} will in general find a local minimum, we performed several solves with different initial points to have a better chance for finding a global minimum.

In all our examples we focus on the non-monotone non-linearity
\[R(y) = y(y-0.25)(y-1)\]
and fix $T=80$. We take $\nu=0$, and impose the bounds
 $0\leq \delay_i\leq T$,
$|\weight_i|\leq 1000$ for $i=1:m$. Figures \ref{F1} and \ref{F2} show the {states} {up to $t = 2T$} to {confirm} that the obtained solutions {exhibit a} stable {behavior} for $t>T$.

\begin{example} \label{Example1}
{We} start with {one example for an ordinary differential delay equation (ODE)}. {This fits} in our setting as long as {$y_0$ and
$R$ are constant with respect to $x$, because then the equation \eqref{E2.1} reduces to an the ODE.} {We consider the ODE with delay}
\begin{equation} \label{ODE1}
y'(t)+R(y) = \sum_{i=1}^m \weight_i y(t-\delay_i)\mbox{ for }t\in (0,T],\ y(t)=y_0(t),\mbox{ if }t\leq 0
\end{equation}
{for $y: [-b,T] \to \mathbb{R}$, where $y_0: [-b,0] \to \mathbb{R}$ is given and $R:  \mathbb{R} \to  \mathbb{R}$ is the given reaction term.}

{We select}  the target state $y_Q$
{solving} the linear delay equation
\[
y'(t) = -\frac{\pi}{2} y(t-1)\mbox{ in }[0,T],\quad  y(t) = 1 {\mbox{ in } [-1,0).}
\]
This function exhibits a stable oscillatory behavior; {displayed as} green curve in Fig. \ref{F1}. {A nice discussion of this particular equation can be found in Erneux \cite{Erneux2009}.}

For $m=1$ and an appropriate choice of the parameters $u=(\delay,\weight)$, we want to mimic that behavior {by} the solution of the {\em nonlinear} delay equation \eqref{ODE1} with initial data $y_0(t)=1$.

{For} the choice $\delay = 1$ and $\weight=-\pi/2$, {the state} exhibits an oscillatory behavior, but {$|y_u|$} decays in time, see the red dashed curve in Fig. \ref{F1}.
{Our optimization problem is to}  minimize
\begin{equation} \label{functional1}
J(\delay,\weight) = \frac{1}{2}\int_0^T (y_u(t)-y_Q(t))^2 dt
\end{equation}
{subject to}  the state equation \eqref{ODE1} and $0\leq \delay\leq T$  and $|\kappa|\leq 1000$.
{Numerically}, we obtained the solution {$\bar u = (\bar \delay,\bar \weight)$ with}
\[\bar \delay = 1.2409   ,\ \bar \weight= -1.7668\]
{and an associated value} $J(\bar u) = 1.8701$
{of the objective functional}. The gradient {of $J$} at the computed solution has the norm $|\nabla J(\bar u)|=3.8\times 10^{-7}$.
Figure \ref{F1} {displays} {the optimal and the desired state in} blue and green respectively. {For comparison}, $y_u(t)$ for $u=(1,-\pi/2)$ {is plotted} in dashed red.
We had to use {$2^{12}$} time steps in the discretization to capture correctly the behaviour of the linear delay equation that defines the target state.
\end{example}

\begin{figure}
  \centering
  \includegraphics[width=\textwidth]{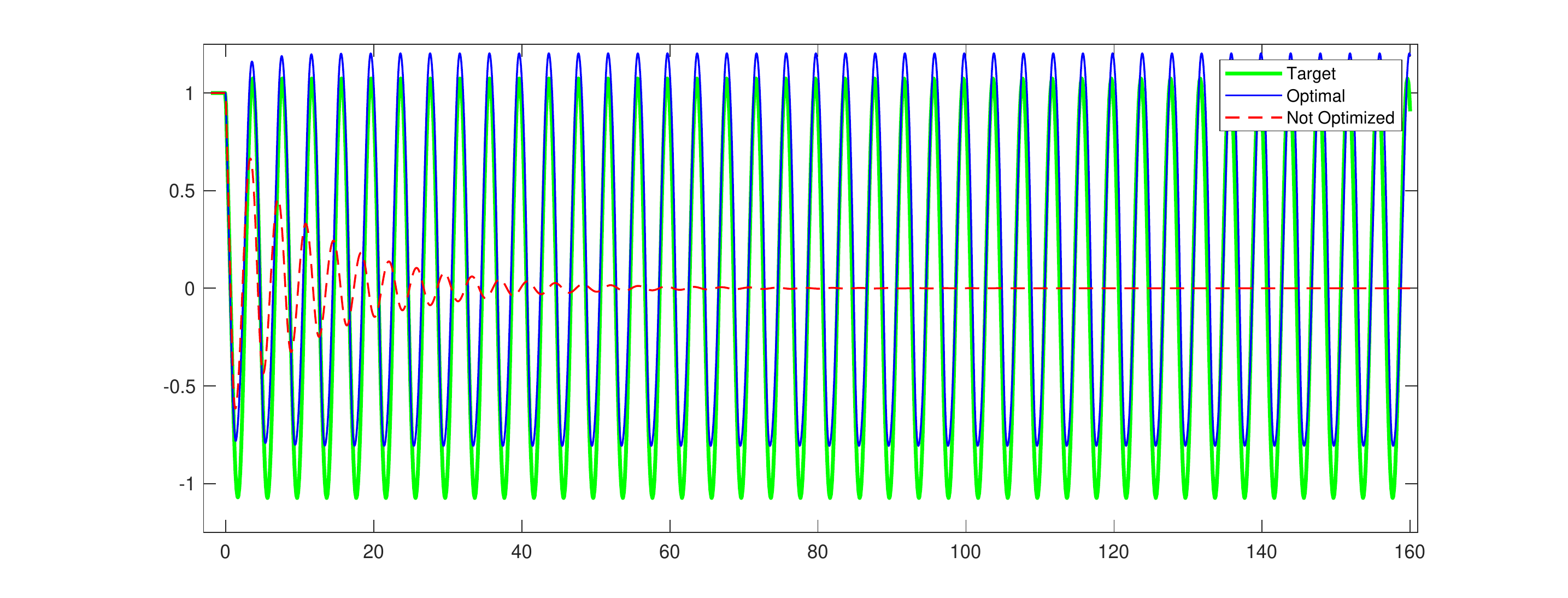}
  \caption{Example \ref{Example1};
  {Target state (green), optimal state (blue), and uncontrolled state (red).}}\label{F1}
\end{figure}

For {all} the next examples, we consider the data of Example 3 in \cite{nestler_schoell_troeltzsch2016}: We fix $\Omega=(-20,20) \subset \mathbb{R}$. The initial {function $y_0$} {models}  an incoming traveling wave, namely
\[
y_0(x,t) =\frac{1}{2}\left[1-\tanh\left(\frac{x-vt}{2}\right)\right],
\]
with $v=0.25\sqrt{2}$. This kind of problems appear in chemical wave propagation; see \cite{Lober2014}. We aim {at steering} the system to the target state shown in Fig. \ref{F2a}
\[
y_Q(x,t) = 3\sin\left(t-\cos\left(\frac{\pi}{20}(x+20)\right)\right).
\]
For the discretization, we take $2^7$ finite elements in space and $2^7$ steps in time.
\begin{example}\label{Example2}
We fix $m=6$ and obtain the optimal parameters shown in Table \ref{Table1}. A graph of the optimal state is shown in Fig. \ref{F2b}.

\begin{table}[h!]
  \[\begin{array}{c|r r}
      i &   \bar\delay_i &\bar\weight_i \\ \hline
      1 &    0.0000 & 0.9846 \\
      2 &    0.9367 & -1.5039 \\
      3 &    6.7481 & 0.4542 \\
      4 &    28.3843 & -2.2799 \\
      5 &    32.2258 & 3.7013 \\
      6 &    39.8133 & -1.3844
    \end{array}
    \]
   \caption{Example \ref{Example2}: Computed optimal result.}\label{Table1}
\end{table}
For these values, we have computed an optimal value  $J(\bar u) = 4209.3$.
{This value is quite large, but note that the measure of $Q = (-20,20)\times (0,80)$ is equal to
$3200$. Therefore, the function $y \equiv 1$ has a norm square of $3200$ in $L^2(Q)$.}

Notice that the lower constraint for the delays
is achieved, since {$\bar s_1=0$}, {and $\sum_{i\neq 1}\kappa_i = -1.0127$, which is quite close to $-\kappa_1$}. {This somehow resembles the original Pyragas feedback form, since the term $y(x,t) = y(x,t-s_1)$ appears in the right-hand side of the partial differential equation, cf. also the subsection on Pyragas type control below}. First order optimality conditions are satisfied: we obtain that $\partial_{s_1}J(\bar u) = 486\geq 0$, remember $\bar s_1$ attains the lower constraint, and the maximum of the absolute value of the rest of the components of the gradient is  $ 2.0\times 10^{-4}$.
\end{example}

\paragraph*{Objective functional with shift in the target}

{If a given periodicity of the state is desired}, then two states with the same period should be considered
as equal if they differ only by a time shift. For instance, the functions $t \mapsto \sin(t)$ and
$t \mapsto \sin(t + \pi)$ should be considered as equal. This is natural, since the time until developing an oscillatory behavior may depend on the selected delays. This inherent shift in time is unavoidable and makes the minimization of standard quadratic tracking type functionals difficult.

{Therefore,} in \cite{nestler_schoell_troeltzsch2016} it was suggested to include a shift $\shift$
in the target state $y_Q$. Then the target can be adjusted to the computed states during the numerical
algorithm.
{In view of this, we will minimize now the shifted functional
\begin{equation}\label{function3}
J(u,\shift)=\frac{1}{2}\int_{0}^{T} (y_u(x,t)-y_Q(x,t-\shift))^2 \,dxdt
\end{equation}
simultaneously with respect to $u \in \U$ and $\shift \in \mathbb{R}$.

{We assume that the desired state
$y_Q$ is time-periodic with period $p > 0$. Then we might impose the additional constraint $\shift \in [0,p]$
that shows the existence of an optimal shift by compactness. However, by periodicity, this constraint can be skipped and is  numerically not needed.
}

{The associated optimality conditions are obtained by minor modification.}
{It is easy to see that,} for given $(u,\shift)$, the adjoint state $\varphi$ is the solution of the equation
\[\left\{\begin{array}{rcll}
-\partial_t \varphi -\Delta \varphi +R'(y_u)\varphi &=& y(x,t)-y_Q(x,t-\shift) 
+ \displaystyle\sum_{i=1}^m \weight_i{\varphi (x,t+\delay_i)}&\mbox{ in }Q\\
\partial_n \varphi(x,t) & = & 0&\mbox{ on }\Sigma\\
\varphi(x,t) & = & 0&\mbox{ if }t\geq T.
\end{array}
\right.
\]

The expressions for the derivatives with respect to the delays and the weights are the same as the ones given in Theorem \ref{T3.3}. The partial derivative with respect to the shift $\shift$ is
\begin{equation}\frac{\partial J}{\partial\shift}(u,\shift) = \int_{0}^{T}\int_\Omega (y_u(x,t)-y_Q(x,t-\shift))\frac{\partial y_Q}{\partial t}(x,t-\shift)\, dxdt.\label{Dwrtsh}\end{equation}

\begin{example}\label{Example3}
We take the same data as in Example \ref{Example2}, fix $m=2$ delays, {and minimize the shifted objective functional  \eqref{function3}. Note that the desired function $y_Q$ has the time period $2\pi$.
}

{The result is} displayed in Table \ref{Table2}, the {computed optimal} state is shown in Fig. \ref{F2c}. {It is amazing, how good the desired pattern is approximated with only two time delays.}

\begin{table}[h!]
  \centering
  \[\begin{array}{c|r r}
      i &   \bar\delay_i & \bar\weight_i \\ \hline
      1 &  2.2785 & -8.2564 \\
      2 &  4.8126 & -5.2898 \\ \hline
      \multicolumn{3}{l}{\mbox{Target shift }\bar \varsigma= 2.3775}
    \end{array}
    \]
  \caption{Example \ref{Example3} (shifted functional): Optimal result}\label{Table2}
\end{table}

In this case, {\texttt{fmincon} computed as optimal value}  $J(\bar u,\bar\varsigma)= 2114.5$  {with  gradient} $|\nabla J(\bar u,\bar\varsigma)| = 1.1\times 10^{-6}$; {It is remarkable that the shift essentially improved the numerical result of
Example \ref{Example2}. Moreover, the computed periodic pattern remains stable after $t = 80$.}

{In \cite{nestler_schoell_troeltzsch2016} it is also suggested to change the objective functional to
\[{\partial J}(u,\shift) = \int_{T/2}^{T}\int_\Omega (y_u(x,t)-y_Q(x,t-\shift))^2\, dxdt\]
because it is reasonable to assume that it takes some time to transfer the incoming traveling wave $y_0$ into
a periodic solution. Using this new functional and }
{increasing the number of time delays to $m=8$, the objective value
can be reduced {down} to $J(\bar u,\bar\shift) = 218.75$.}

\end{example}

\paragraph*{Pyragas type feedback control}
Finally, we investigate the approximation of oscillatory patterns that are characteristic for Pyragas type feedback control {as in \eqref{E1.1},}
\begin{equation}\label{E4.1}
\left\{\begin{array}{rcll}
\partial_t y -\Delta y +R(x,t,y) &=& \displaystyle\sum_{i=1}^m \weight_i (y(x,t-\delay_i)-y(x,t))&\mbox{ in }Q\\
\partial_n y & = & 0&\mbox{ on }\Sigma\\
y(x,t) & = & y_0(x,t)&\mbox{ in }Q^-.
\end{array}
\right.
\end{equation}
We want to design a feedback controller by adjusting finitely many time delays and associated weights minimizing the shifted functional \eqref{function3}.

The adjoint state equation in this case is
\[\left\{\begin{array}{rcll}
-\partial_t \varphi -\Delta \varphi +R'(y_u)\varphi &=& y(x,t)-y_Q(x,t-\shift) \\
&&
+ \displaystyle\sum_{i=1}^m \weight_i\left(\varphi (x,t+\delay_i)-\varphi (x,t)\right)&\mbox{ in }Q\\
\partial_n \varphi(x,t) & = & 0&\mbox{ on }\Sigma\\
\varphi(x,t) & = & 0&\mbox{ if }t\geq T.
\end{array}
\right.
\]
The expressions for the derivatives with respect to the delays and the shift are {the same as} the ones given in equations \eqref{Dwrtdelay} and \eqref{Dwrtsh}, while the derivative with respect to the weight is given by the expression
\[\frac{\partial J}{\partial \weight_i}(u) = \nu \weight_i +\int_Q\varphi_u(x,t) (y_u(x,t-\delay_i)-y_u(x,t)) \,dxdt.\]
\begin{example}\label{Example4}
With the same data as in examples \eqref{Example2}, we fix $m=4$ and obtain the optimal parameters shown in Table \ref{Table3}. A  plot of the optimal state is displayed in Fig. \ref{F2d}.
For {these} values, we  computed {an} optimal value  $J(\bar u,\bar\shift) = 3763.4$ with $|\nabla J(\bar u,\bar\shift)| = 4.8\times 10^{-4}$.

\begin{table}[h!]
  \centering
  \[\begin{array}{c|r r}
      i &   \bar\delay_i &\bar\weight_i \\ \hline
      1 &   1.8308 & -2.1661 \\
      2 &   7.0918 & 2.2636 \\
      3 &   28.3354 & -1.7753 \\
      4 &   36.1215 & 1.7550 \\ \hline
      \multicolumn{3}{l}{\mbox{Target shift }\bar \varsigma= -2.5013}
    \end{array}
    \]
  \caption{Example \ref{Example4}: Computed optimal result.}\label{Table3}
\end{table}

\end{example}

\begin{figure}
  \centering
  \subfloat[Target]{\label{F2a}\includegraphics[height=.35\textheight]{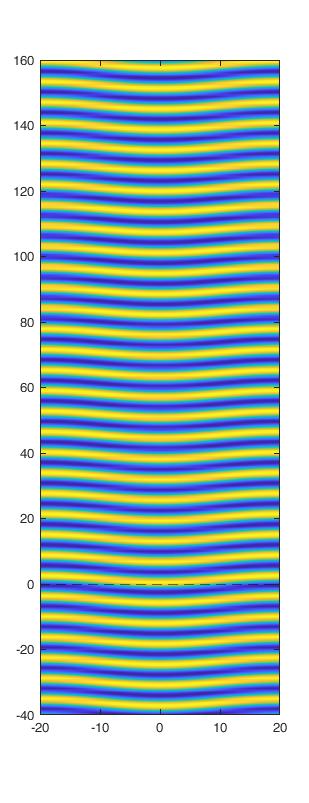}}
  \subfloat[Example \ref{Example2}]{\label{F2b}\includegraphics[height=.35\textheight]{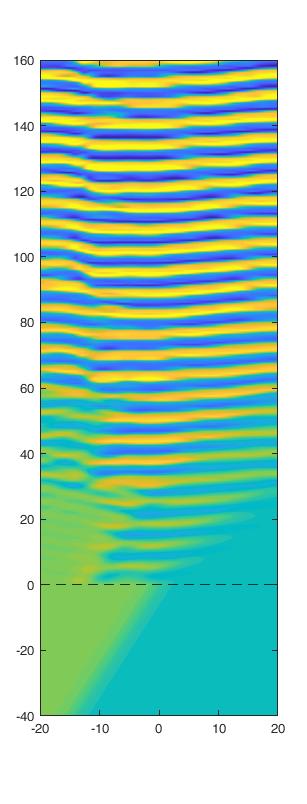}}
  \subfloat[Example \ref{Example3}]{\label{F2c}\includegraphics[height=.35\textheight]{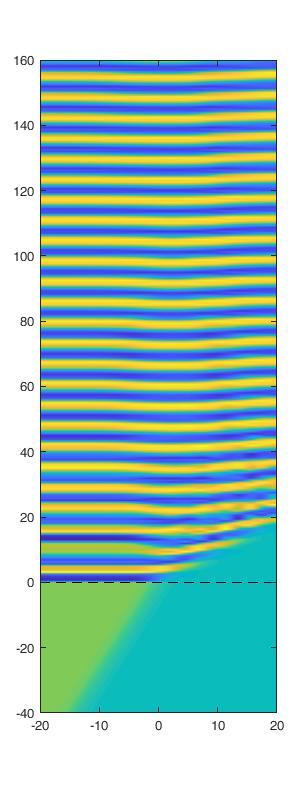}}
  \subfloat[Example \ref{Example4}]{\label{F2d}\includegraphics[height=.35\textheight]{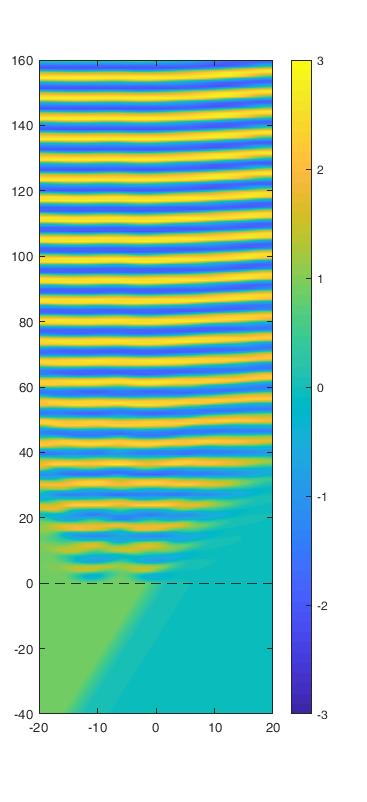}}
  \caption{{Examples \ref{Example2}-\ref{Example4}:} Target and optimal states. All functions are shown in $\Omega\times[-T/2,2T]$.}\label{F2}
\end{figure}

\section*{Acknowledgments}

The first two authors were partially supported by Spanish Ministerio de Econom\'{\i}a y Competitividad under research projects MTM2014-57531-P and MTM2017-83185-P. The third author was supported by the collaborative
research center SFB 910, TU Berlin, project B6.


\end{document}